\documentclass[12pt]{article}%
\usepackage{amsfonts}
\usepackage{amsmath}
\usepackage{amssymb}
\usepackage{graphicx}%
\providecommand{\U}[1]{\protect\rule{.1in}{.1in}}
\newtheorem{theorem}{Theorem}[section]
\newtheorem{condition}[theorem]{Condition}
\newtheorem{corollary}[theorem]{Corollary}

\newtheorem{lemma}[theorem]{Lemma}

\newtheorem{proposition}[theorem]{Proposition}
\newtheorem{remark}[theorem]{Remark}

\newenvironment{proof}[1][Proof]{\noindent\textbf{#1.} }{\ \rule{0.5em}{0.5em}}
\begin{document}

\title{A note on the policy iteration algorithm for discounted Markov decision
processes for a class of semicontinuous models\thanks{In memoriam of our
friend and colleague Rolando Cavazos-Cadena who passed away last May.}}
\author{\'{O}scar Vega-Amaya\thanks{Correspondence author. Email: ovega@mat.uson.mx}
and Fernando Luque-V\'{a}squez\thanks{Email: fluque@mat.uson.mx}\\Departamento de Matem\'{a}ticas\\Universidad de Sonora\thanks{Luis Encinas y Abelardo L. Rodr\'{\i}guez s/n, C.
P. 83000, Hermosillo, Sonora, M\'{e}xico}}
\date{July 05, 2023}
\maketitle

\begin{abstract}
The standard version of the policy iteration (PI) algorithm fails for
semicontinuous models, that is, for models with lower semicontinuous one-step
costs and weakly continuous transition law. This is due to the lack of
continuity properties of the discounted cost for stationary policies, thus
appearing a measurability problem in the improvement step. The present work
proposes an alternative version of PI algorithm which performs an smoothing
step to avoid the measurability problem. Assuming that the model satisfies a
Lyapunov growth conditions and also some standard continuity-compactness
properties, it is shown the linear convergence of the policy iteration
functions to the optimal value function. Strengthening the continuity
conditions, in a second result, it is shown that among the improvement
policies there is one with the best possible improvement and whose cost
function is continuous.


\end{abstract}

\section{Introduction}

The policy iteration (PI) algorithm is a popular procedure for solving optimal
control problems \cite{Boka, HL96, HL23, Put94}. It is credited to R. A.
Howard \cite{Howard}, so it is also known as Howard's improvement algorithm or
just Howard's algorithm. Since its inception, researchers have been interested
in the PI algorithm because its good performance in many specific problems
\cite{Santos} or because it serves as a base for other numerical procedures,
and as well by its close relationship with two other important algorithms,
namely, the simplex method of linear programming \cite{Put94, Wessels, Ye} and
the Newton-Raphson method \cite{Bert22, Ohnishi, Put79, Put94}. However, the
PI algorithm experiencies a measurability problem for lower semicontinuous
models, that is, for models with lower semicontinuous one-step cost and weakly
continuous transition law; in fact, this measurability difficulty is also
present in the analytical or universal measurabilty framework. Yu and
Bertsekas \cite{Yu15} give a very detailed discussion of this issue.

Recall that the PI algorithm runs iteratively two steps: the first one is the
evaluation step, which finds the cost function of a given stationary policy;
the second one is the improvement step, which finds a measurable selector (or
stationary policy) that reaches the minimum in the dynamic programming
operator acting on the cost function found previously, and then the algorithm
goes back to the first step, and so on.

The measurability problem appears in the second step because the lack of
continuity properties of the cost function given by the previous step. One way
to overcome this measurability problem is to consider variants in which the
improvement is performed respect to a different function than the one coming
from the first step; in general, it is expected the convergence analysis turns
out much more involved than in the standard case. Yu and Bertsekas \cite{Yu15}
developed a mixed value and policy iteration algorithm for models with Borel
spaces and universally measurable policies, in general, and for lower
semicontinuous models too. The present note focuses on semicontinuous models,
but instead of combining policy iteration with value iteration--or any other
approximation scheme--it performs first a smoothing or regularization step
over the previous cost function which finds its lower semicontinuous envelope.
Thus, the improvement step is done over this latter semicontinuos function.
The convergence analysis of this variant of PI\ algorithm is straightforward
as it is in the standard case.

Specifically, this note shows the convergence of this variant of the PI
algorithm for a class of discounted optimal control problems for lower
semicontinuous models with compact admissible actions sets and assumming that
the one-step cost function and the transition law satisfy a Lyapunov\ growth
condition. In this framework the dynamic programming operator is a
contraction, which allows to show that the policy iteration functions converge
linearly to the optimal value function in a weighted norm( see Theorem
\ref{PIA} below); in particular, for bounded cost function the convergence is
uniform. Strengthening the conditions by assumming that the one-step cost
function and the admissible actions set multifunction are continuous, it is
shown that among the improvement policies there is one with the best possible
improvement and whose cost function is continuous (Theorem \ref{MPI} below).
However, finding such a policy requires solving a new optimal control problem,
which is expected to be simple one or not too complicated in specific problems.

\section{Markov decision model}

Consider the standard Markov decision model $(X,A,\{A(x):x\in X\},Q,C)$ where:
(a) $X$ and $A$ denote the state and control (or action) spaces; both $X$ and
$A$ are Borel spaces, that is, Borel subsets of a complete separable metric
spaces; (b) for each $x\in X,$ $A(x)$ is a subset of $A$ and denotes the
admissible actions for the state $x;$ the admissible pair state-action pairs
set $\mathbb{K}:=\{(x,a):x\in X,a\in A(x)\}$ is assumed to be a Borel subset
of the cartesian product $X\times A;$ (c) the transition law $Q$ is an
stochastic kernel on $X$ given $\mathbb{K},$ that is, $Q(\cdot|x,a)$ is a
probability measure on $X$ for each $(x,a)\in\mathbb{K},$ and $Q(B|\cdot
,\cdot)$ is a (Borel) measurable function on $\mathbb{K}$ for each (Borel)
measurable subset $B\subset X$;\ (d) the one-step cost $C:\mathbb{K\rightarrow
R}$ is a (Borel) measurable function; $\mathbb{R}$ stands for the set of real numbers.

The fifth-tuple $(X,A,\{A(x):x\in X\},Q,C)$ models a controlled system that
evolves as follows: at time $t=0,$ the controller or decision maker observes
the initial system state $x_{0}=x\in X$ and chooses a control or decision
$a_{0}=a\in A(x)$ incurring in a cost $C(x,a);$ then, the system changes to
the state $x_{1}=y\in X$ according to the probability measure $Q(\cdot|x,a),$
that \ is, $\Pr[x_{1}\in B|x_{0}=x,a_{0}=a]=Q(B|x,a)$ for mesurable subsets
$B\subset X.$ After that, the controller choose and action $a_{1}=b\in A(y)$
with a cost $C(y,b)$ and so on. Thus, let $x_{n}$ and $a_{n}$ be the state and
the control at time $n\in\mathbb{N}_{0}.$

Let $H_{0}:=X$ and $H_{n}:=\mathbb{K}\times H_{n-1}$ for $n\in\mathbb{N}.$
Thus, the history of the systems up to time $n\in\mathbb{N}_{0}$ is given by
$h_{n}=(x_{0},a_{0},x_{1},a_{1},\ldots,x_{n-1},a_{n-1},x_{n})\in H_{n}$. A
control policy is a sequence $\pi=\{\pi_{n}\}$ where each $\pi_{n}$ is an
stochastic kernel on $A$ given $\mathbb{H}_{n}$ satisfying the condition
$\pi_{n}(A(x_{n})|h_{n})=1$ for all $h_{n}\in H_{n}.$The class of all policies
is denoted by $\Pi.$

Denote by $\mathbb{F}$ the class of all measurable selectors from $X$ to $A,$
that is, the measurable functions $f:X\rightarrow A$ that satisfies the
condition $f(x)\in A(x)$ for all $x\in X.$ A control policy $\pi=\{\pi_{n}\}$
is said to be stationary if for some $f\in\mathbb{F}$ the measure $\pi
_{n}(\cdot|h_{n})$ is concentrated at $f(x_{n})$ for all $h_{n}\in
H,n\in\mathbb{N}_{0}.$ In this case, the control policy $\pi=\{\pi_{n}\}$ is
identified with the selector $f$ and the class of all stationary policies is
identified with $\mathbb{F}.$

Let $\Omega:=(X\times A)^{\infty}$ and $\mathcal{F}$ the corresponding product
$\sigma$-algebra. It is well-known that for each policy $\pi\in\Pi$ and
\textquotedblleft initial\textquotedblright\ state $x\in X$ there exists a
unique probability measure $P_{x}^{\pi}$ on the measurable space
$(\Omega,\mathcal{F)}$ such that the following properties hold for all
$n\in\mathbb{N}_{0}$ : (a) $P_{x}^{\pi}[x_{0}=x]=1;\ $(b) $P_{x}^{\pi}%
[a_{n}\in D|h_{n}]=\pi_{n}(D|h_{n})$ for all measurable subset $D\subset A;$
(c) $P_{x}^{\pi}[x_{n+1}\in B|h_{n}]=Q(B|x_{n},a_{n})$ for all measurable
subsets $B\subset X.$

Let \ $\alpha\in(0,1)$ be a fixed \textquotedblleft discount
factor\textquotedblright. The ($\alpha$-)discounted cost for policy $\pi
=\{\pi\}\in\Pi$ and initial state $x_{0}=x\in X$ is defined as%
\[
V_{\pi}(x):=E_{x}^{\pi}\sum_{k=0}^{\infty}\alpha^{k}C(x_{k},a_{k}).
\]
The discounted optimal value function is given as%
\[
V(x):=\inf_{\pi\in\Pi}V_{\pi}(x),\ \ x\in X.
\]
Thus, a policy $\pi^{\ast}=\{\pi_{n}^{\ast}\}\in\Pi$ is said to be optimal if%
\[
V(x)=V_{\pi^{\ast}}(x)\ \ \forall x\in X.
\]

The following notation is used throughtout of the remainder of this note:%
\[
Qu(x,a):=\int_{X}u(y)Q(dy|x,a)
\]
for $(x,a)\in\mathbb{K}$ and functions $u:X\rightarrow\mathbb{R}$ for which
the integral is well defined. Moreover, for a policy $f\in\mathbb{F}$ let%
\[
C_{f}(x):=C(x,f(x))\ \ \ \text{and\ \ \ }Q_{f}u(x):=Qu(x,f(x)),\ \ x\in X.
\]

Next, for an arbitrary function $u:X\rightarrow\mathbb{R}$ define%
\[
u^{e}(x):=\sup_{r>0}\inf_{y\in B_{r}(x)}u(y),
\]
where $B_{r}(x)$ stands for the open ball centered in $x\in X$ with radius
$r>0.$ Note that $u^{e}$ is the largest lower semicontinuous function
dominated by $u$, that is, if a function $v:X\rightarrow\mathbb{R}$ is lower
semicontinuos and $u\geq v$, then $u^{e}\geq v.$ Thus, $u^{e}$ is called the
lower semicontinuous envelope of function $u.$ Moreover, note that $u$ is
lower semicontinuous if and only if $u=u^{e}.$

The proof of the convergence of the policy iteration functions uses the
following result on the interchange of minimum and limit.

\begin{proposition}
\label{lim}(\cite[Note 5, p. 53]{HL96}) Let $u_{n},u:\mathbb{K\rightarrow R}$
be measurable functions. If $u_{n}\downarrow u,$ then%
\[
\lim_{n\rightarrow\infty}\inf_{a\in A(x)}u_{n}(x,a)=\inf_{a\in A(x)}u(x,a).
\]

\end{proposition}

\section{Lyapunov condition and preliminary results}

This section gather some important known results for the discounted optimal
control problem obtained assumming the control model satisfies two standard
set of conditions. The first one imposes the next growth Lyapunov condition on
the costs and also on the evolution law.

\begin{condition}
\label{GC}Growth conditions: There exist constants $M>0,\beta>1$ and a
function $W:X\rightarrow\lbrack1,\infty)$ such that:\smallskip

(a) $|C(x,a)|\leq MW(x)$ for all $(x,a)\in\mathbb{K};$\smallskip

(b) $QW(x,a)=\int_{X}W(y)Q(dy|x,a)\leq\beta W(x)$ for all $(x,a)\in
\mathbb{K};$\smallskip

(c) $\gamma:=\beta\alpha<1.$
\end{condition}

The second set of condition concerns with usual continuity/compactness properties.

\begin{condition}
\label{CC}Compactness-continuity conditions:\smallskip

(a) the mapping $x\rightarrow A(x)$ is upper-semicontinuous and
compact-valued;\smallskip

(b) $C$ is lower semicontinuous on the set $\mathbb{K}$;\smallskip

(c) $Q(\cdot|\cdot,\cdot)$ is weakly continuous on $\mathbb{K},$ that is, the
mapping%
\[
(x,a)\rightarrow Qu(x,a):=\int_{X}u(y)Q(dy|x,a)
\]
is continuous for each bounded continuous function $u:X\mathbb{\rightarrow
}\mathbb{R};$\smallskip

(d) the functions $W$ and $QW$ are continuous.
\end{condition}

Now let $B_{W}(X)$ be the class of functions $u:X\rightarrow\mathbb{R}$ such
that%
\[
||u||_{W}:=\sup_{x\in X}\frac{|u(x)|}{W(x)}<\infty.
\]
Denote by $L_{W}(X)$ and $C_{W}(X)$ the subclasses of functions of $B_{W}(X)$
that are lower semicontinuous and continuous, respectively. Notice that
$B_{W}(X)$ and $C_{W}(X)$ are Banach spaces and also that $L_{W}(X)$ is a
complete metric subspace with respect to the metric induced by the norm
$||\cdot||_{W}.$ Moreover, observe that if the function $W$ is continuous,
then $u^{e}\in L_{W}(X)$ for all $u\in B_{W}(X).$

Note that under Condition \ref{GC}, for each policy $\pi\in\Pi$ it holds that
$|V_{\pi}|\leq M(1-\gamma)^{-1}W.$ Thus, $|V_{\ast}|\leq M(1-\gamma)^{-1}W$.
Hence, the functions $V_{\ast}$ and $V_{\pi},\pi\in\Pi,$ belong to $B_{W}(X)$.

Next for each $f\in\mathbb{F}$ define%
\[
T_{f}u(x):=C_{f}(x)+\alpha Q_{f}u(x),\ \ x\in X,
\]
for functions $u\in B_{W}(X)$. The dynamic programming operator is defined as%
\[
Tu(x):=\inf_{a\in A(x)}[C(x,a)+\alpha Qu(x,a)],\ \ x\in X.
\]

For the proof of all results of this section the reader is referred to
\cite[Section 8.5, p. 65]{HL99}.

\begin{remark}
\label{FP1}Suppose that Conditions \ref{GC} and \textbf{\ref{CC}} hold. Then:

(a) For each function $u\in L_{W}(X)$ the function $Tu\in L_{W}(X)$ and there
exists $f\in\mathbb{F}$ such that%
\[
Tu=T_{f}u=C_{f}+\alpha Q_{f}u.
\]
For this results see, for instance, \cite[Remark 3.6]{OVA18}.

(b) Moreover, $T$ is a contraction operator on $L_{W}(X)$ with modulus
$\gamma.$

(c) Similarly, $T_{f},f\in\mathbb{F},$ is a contraction operator from
$B_{W}(X)$ into itself with contraction modulus $\gamma$ and $V_{f}$ is the
unique fixed point of $T_{f}$ in $B_{W}(X),$ that is,%
\[
V_{f}=T_{f}V_{f}=C_{f}+\alpha Q_{f}V_{f}.
\]

\end{remark}

\begin{theorem}
\label{FP2} Suppose that Conditions \ref{GC} and \ref{CC} hold. The optimal
value function $V_{\ast}$ is the unique fixed point of operator $T$ in
$L_{W}(X),$ that is, $V_{\ast}$ is the unique function in $L_{W}(X)$ that
satisfies the optimality equation%
\[
V_{\ast}(x)=\inf_{a\in A(x)}[C(x,a)+\alpha\int_{X}V_{\ast}%
(y)Q(dy|x,a)]\ \ \forall x\in X.
\]
Thus:

(a) there exists $f^{\ast}\in\mathbb{F}$ such that $TV_{\ast}=T_{f_{\ast}%
}V_{\ast};$

(b) a policy $f$ $\in\mathbb{F}$ is optimal if and only if $TV_{\ast}%
=T_{f}V_{\ast};$ hence, the policy $f^{\ast}$ is optimal.
\end{theorem}

The readers can found in reference \cite[Section 8.6, p. 68]{HL99} an
inventory model and a queueing system that satisfy the assumptions in Theorem
\ref{MPI}.

\section{The policy iteration algorithm}

In this section it is assumed that Conditions \ref{GC} and \ref{CC} hold.
Recall that $V_{f}^{e},f\in\mathbb{F},$ stands for the lower semicontinuous
envelope of fuction $V_{f}$ and note that $V_{f}^{e}$ belongs to $B_{W}(X)$
since $W$ is continuous. Thus, the policy iteration algorithm runs as follows.

\begin{description}
\item[Initial step.] Set $n=0$ and choose $f_{n}\in\mathbb{F}.$

\item[Evaluation step.] Find the function $v_{n}:=V_{f_{n}}.$ By Remark
\ref{FP1}(c), this function can be found by solving the equation%
\[
v=C_{f_{n}}+\alpha Q_{f_{n}}v.
\]

\item[Smoothing/regularization step.] Find $v_{n}^{e}.$

\item[Improvement step.] Find a selector $f_{n+1}\in\mathbb{F}$ such that%
\begin{align*}
Tv_{n}^{e}(x)  &  =\min_{a\in A(x)}[C(x,a)+\alpha\int_{X}v_{n}^{e}%
(y)Q(dy|x,a)]\\
& \\
&  =C_{f_{n+1}}(x)+\alpha Q_{f_{n+1}}v_{n}^{e}(x)
\end{align*}
for all $x\in X$. Next, put $n:=n+1$ and go to the evaluation step. Note that
Remark \ref{FP1}(a) ensures the existence of such selector.
\end{description}

The next lemma proves the basic facts for the convergence of the policy
iteration functions $\{v_{n}\}.$

\begin{lemma}
\label{L1} Let $g_{0}\in\mathbb{F}$ be an arbitrary selector and $g_{1}%
\in\mathbb{F}$ such that%
\[
TV_{g_{0}}^{e}=T_{g_{1}}V_{g_{0}}^{e}.
\]

Then:\smallskip

(a) $V_{g_{0}}^{e}\geq TV_{g_{0}}^{e};$\smallskip

(b) $V_{g_{0}}^{e}\geq V_{g_{1}};$\smallskip

(c) $V_{g_{0}}^{e}\geq TV_{g_{0}}^{e}\geq V_{g_{1}};$\smallskip

(d) if $V_{g_{1}}=V_{g_{0}}^{e}$ then $V_{g_{1}}=V_{\ast}.$\smallskip


\end{lemma}

\begin{proof}
(a) Observe that%
\[
V_{g_{0}}=T_{g_{0}}V_{g_{0}}\geq T_{g_{0}}V_{g_{0}}^{e}\geq TV_{g_{0}}^{e}.
\]
Since $TV_{g_{0}}^{e}$ is lower semicontinuous, it follows that%
\[
V_{g_{0}}^{e}\geq TV_{g_{0}}^{e}.
\]

(b) Recall that $TV_{g_{0}}^{e}=T_{g_{1}}V_{g_{0}}^{e}.$ This and part (a)
implies that
\[
V_{g_{0}}^{e}\geq T_{g_{1}}^{n}V_{g_{0}}^{e}.
\]
Remark \ref{FP1}(c) implies that $T_{g_{1}}^{n}V_{g_{0}}^{e}\rightarrow
V_{g_{1}}$ in the weighted norm $||\cdot||_{W},$ and thus pointwise too.
Hence, the result follows after taking limit in both sides of the above inequality.

(c) These inequalities follows from parts (a), (b) and Remark \ref{FP1}(c):%
\[
V_{g_{0}}^{e}\geq TV_{g_{0}}^{e}=T_{g_{1}}V_{g_{0}}^{e}\geq T_{g_{1}}V_{g_{1}%
}=V_{g_{1}}.
\]

(d) This part is a direct consequence of part (c) and Theorem \ref{FP2}.


\end{proof}

\begin{theorem}
\label{PIA}Suppose that Assumption \ref{GC} holds. Let $f_{0}\in\mathbb{F}$ be
an arbitrary stationary policy and $\{v_{_{n}}\}$ be the sequence of policy
iteration functions starting with policy $f_{0}$. Then:\smallskip

(a) $v_{n}\geq v_{n}^{e}\geq Tv_{n}^{e}\geq v_{n+1}\geq v_{n+1}^{e}$ for all
$n\in\mathbb{N}_{0};$\smallskip

(b) if $v_{n}^{e}=v_{n+1},$ then $v_{n}^{e}=V_{\ast}$ and $f_{n+1}$ is an
optimal policy;\smallskip

(c) $V_{\ast}=\lim_{n\rightarrow\infty}v_{n}=\lim_{n\rightarrow\infty}%
v_{n}^{e}=\lim_{n\rightarrow\infty}Tv_{n}^{e}.$
\end{theorem}

\begin{proof}
Parts (a) and (b) follow directly from Lemma \ref{L1}. To prove part (c) first
observe that%
\[
w:=\lim_{n\rightarrow\infty}v_{n}=\lim_{n\rightarrow\infty}v_{n}^{e}%
=\lim_{n\rightarrow\infty}Tv_{n}^{e}.
\]
Clearly, $w\geq V_{\ast}$; moreover, it belongs to $B_{W}(X)$ because the
inequalities%
\[
-M(1-\gamma)^{-1}W\leq v_{n}\leq M(1-\gamma)^{-1}W
\]
hold for all $n\in\mathbb{N}$. Then, from Proposition \ref{lim}, it follows
that $w=Tw.$ This implies that%
\[
w(x)\leq C(x,a)+Qw(x,a)\ \ \forall(x,a)\in\mathbb{K},
\]
which in turn leads to%
\[
w(x)\leq E_{x}^{\pi}\sum_{k=0}^{n-1}\alpha^{k}C(x_{k},a_{k})+\alpha^{n}%
E_{x}^{\pi}w(x_{n})
\]
for all $x\in X,\pi\in\Pi$ and $n\in\mathbb{N}_{0}.$ On the other hand,
Condition \ref{GC} implies that%
\[
\alpha^{n}E_{x}^{\pi}w(x_{n})\leq||w||_{W}(\alpha\beta)^{n}W(x)\ \ \forall
x\in X,\pi\in\Pi.
\]
Thus,%
\[
w(x)\leq V_{\pi}(x)\ \ \forall x\in X,\pi\in\Pi,
\]
which yields that%
\[
w(x)\leq V_{\ast}(x)\ \ \forall x\in X.
\]
Therefore, $w=V_{\ast}.$
\end{proof}

The next result extends \cite[Thm. 6.4.6, p. 180]{Put94} establishing the
linear convergence of the PI algorithm for unbounded costs.

\begin{corollary}
\label{LC} Suppose that Conditions \ref{GC} and \ref{CC} hold. The sequence
$\{v_{n}\}$ converges linearly to $V_{\ast}$ in the $W$-norm. In fact,%
\[
||v_{n+1}-V_{\ast}||_{W}\leq\gamma||v_{n}-V_{\ast}||_{W}\leq L\gamma^{n}%
\]
for all $n\in\mathbb{N},$ with $L:=||v_{0}-V_{\ast}||_{W}.$ Thus,%
\[
\limsup_{n\rightarrow\infty}\frac{||v_{n+1}-V_{\ast}||_{W}}{||v_{n}-V_{\ast
}||_{W}}\leq\gamma.
\]

\end{corollary}

\begin{proof}
Theorems \ref{PIA}(a) and \ref{FP2}, and Remark \ref{FP1} imply that%
\[
||v_{n+1}-V_{\ast}||_{W}\leq||Tv_{n}^{e}-TV_{\ast}||_{W}\leq\gamma||v_{n}%
^{e}-V_{\ast}||_{W}\leq\gamma||v_{n}-V_{\ast}||_{W}
\]
for all $n\in\mathbb{N}$. The desired results follows immediately from the
latter inequality.
\end{proof}

\begin{remark}
(a) From Corollary \ref{LC}, the policy iteration funtions $\{v_{n}\}$
converges uniformly to $V_{\ast}$ on the sublevel sets $X_{\lambda}:=\{x\in
X:W(x)\leq\lambda\},\lambda\in\mathbb{R}$. If the function $W$ is
inf-compact--that is, the sets $X_{\lambda},\lambda\in\mathbb{R},$ are
compact--the sequence converges uniformly on compact sets.

(b) On the other\ hand, if the cost function is bounded, Condition \ref{GC}
holds trivially with $M:=\sup_{(x,a)\in\mathbb{K}}|C(x,a)|$ and $W\equiv1.$
Thus, Theorem \ref{LC} implies that the policy iteration functions $\{v_{n}\}$
converge uniformly to the optimal value function $V_{\ast}.$
\end{remark}

If Assumptions \ref{CC} is strengthened by additionally assuming that
the\ cost function $C$ is continuous on $\mathbb{K}$ and that the mapping
$x\rightarrow A(x)$ is continuous too, it is possible to choose stationary
policies in the improvement step whose discounted costs are lower
semicontinuous functions. However, finding such policies requires the solution
of a new control problem, which will not be difficult to solve in specific
problems because there are usually a small number of policies that solves the
improvement step; in fact, in many cases, there is only one improvement
policy. This result is shown in the next theorem, which uses the following
notation. For a function $v\in B_{W}(X)$ let%
\[
Lv(x,a):=C(x,a)+\alpha Qv(x,a),\ \ (x,a)\in\mathbb{K},
\]
and for a policy $f_{0}\in\mathbb{F}$ define
\[
A_{1}(x):=\{a\in A(x):TV_{f_{0}}^{e}(x)=LV_{f_{0}}^{e}(x,a)\},\ \ x\in X,
\]
and%
\[
\mathbb{F}_{1}:=\{f\in\mathbb{F}:f(x)\in A_{1}(x)\text{ for all }x\in X\}.
\]
Note that $\mathbb{F}_{1}$ is the family of all the improvement policies
generated by the function $V_{f_{0}}^{e}.$

\begin{theorem}
\label{MPI}Suppose that Conditions \ref{GC} and \ref{CC} hold and let
$f_{0}\in\mathbb{F}$ be an arbitrary stationary policy. If in addition the
one-step cost function $C$ is continuous on $\mathbb{K}$ and the mapping
$x\rightarrow A(x)$ is continuous, then:

(a) $TV_{f_{0}}^{e}$ is continuous;

(b) the mapping $x\rightarrow A_{1}(x)$ is compact-valued and upper semicontinuous;

(c) there exists a unique continuous function $w_{1}\in C_{W}(X)$ and a
stationary policy $f_{1}\in\mathbb{F}_{0}$ such that%
\begin{align*}
w_{1}(x)  &  =\inf_{a\in A_{1}(x)}[C(x,a)+\alpha Qw_{1}(x,a)]\\
& \\
&  =C_{f_{1}}(x)+\alpha Q_{f_{1}}w_{1}(x)
\end{align*}
for all $x\in X;$

(d) thus, $f_{1}$ is the best improvement policy, that is,
\[
V_{f_{1}}(x)=\inf_{f\in\mathbb{F}_{1}}V_{f}(x)\ \ \ \forall x\in X.
\]

\end{theorem}

\begin{proof}
First note that under Conditions \ref{GC} and \ref{CC}, the mapping%
\[
(x,a)\rightarrow\int_{X}v(y)Q(dy|x,a)
\]
is continuous for any function $u\in L_{W}(X)$ \cite[Lemma 8.5.5 (a)]{HL99}.
Thus, if $C$ is continuous on $\mathbb{K},$ the function $Lv$ is also
continuous on $\mathbb{K}.$ Moreover, because the mapping $x\rightarrow A(x)$
is compact-valued upper semicontinuous, for each $s\in S$ and sequence
$\{s_{n}\}\subset X$ such that $s_{n}\rightarrow s,$ any actions sequence
$a_{n}\in A(s_{n}),n\in\mathbb{N},$ has an accumulation point $a\in A(s).$
Then, by \cite[Lemma 2.5 and Theorem 4.1]{Fein13}, $TV_{f_{0}}^{e}$ is a
continuous function and $x\rightarrow A_{1}(x)$ is a compact-valued upper
semicontinuous mapping.

Now consider the discounted optimal control problem for the Markov decision
model $(X,A,\{A_{1}(x):x\in X\},Q,C).$ By Theorem \ref{FP2} and \cite[Lemma
2.5 and Theorem 4.1]{Fein13}, there exist a unique function $w_{1}$ and policy
$f_{1}\in\mathbb{F}_{0}$ for which parts (c) and (d) hold.
\end{proof}

Under conditions of this theorem, the PI algorithm reads as follows.

\begin{description}
\item[Initial step.] Choose $f_{0}\in\mathbb{F}$ and find $w_{0}:=V_{f_{0}%
}^{e}$.

\item[First improvement step.] Set $n=0$ and find the set%
\[
A_{n+1}(x):=\{a\in A(x):Lw_{n}(x,a)=Tw_{n}(x)\},\ \ x\in X.
\]

\item[Best improvement and evaluation steps.] Find a policy $f_{n+1}%
\in\mathbb{F}$ and a function $w_{n+1}\in C_{W}(X)$ satisfying the equations%
\begin{align}
w_{n+1}(x)  &  =\inf_{a\in A_{n+1}(x)}[C(x,a)+\alpha Qw_{n+1}(x,a)]\label{NOE}%
\\
& \nonumber\\
&  =C_{f_{n+1}}(x)+\alpha Q_{f_{n+1}}w_{n+1}(x)\nonumber
\end{align}
for all $x\in X.$ Next, put $n:=n+1$ and go to the first improvement step.
\end{description}

The linear convergence of the policy iteration functions $\{w_{n}\}$ is
established in the following corollary, which is a direct consequence of
Theorems \ref{PIA} and \ref{MPI}.

\begin{corollary}
Suppose conditions of Theorem \ref{MPI} hold. Then:

(a) if $w_{n}=w_{n+1}$, then $w_{n}=V_{\ast}$ and $f_{n+1}$ is optimal;

(b) $||w_{n+1}-V_{\ast}||_{W}\leq||w_{n}-V_{\ast}||_{W}\gamma\leq L\gamma^{n}$
for all $n\in\mathbb{N},$ with $L^{\prime}:=||w_{0}-V_{\ast}||_{W};$

(c) thus,%
\[
\limsup_{n\rightarrow\infty}\frac{||w_{n+1}-V_{\ast}||_{W}}{||w_{n}-V_{\ast
}||_{W}}\leq\gamma.
\]

\end{corollary}

\begin{remark}
(a) The inventory model and the queueing systems given \cite[Section 8.6, p.
68]{HL99} satisfies the conditions in Corollaries \ref{LC} and \ref{MPI}.

(b) Finally notice that if the set-valued mappings $x\rightarrow A_{n+1}(x)$
admits just one measurable selector $f_{n+1}\in\mathbb{F}$ for each
$n\in\mathbb{N}_{0},$ then the latter PI algorithm becomes the standard policy
iteration algorithm. This happens, for instance, if the minimizers in
(\ref{NOE}) are unique for each $x\in X$ and $n\in\mathbb{N}_{0}.$
\end{remark}

\end{document}